\documentclass[a4paper,12pt]{article}
\usepackage[T2A]{fontenc}
\usepackage[english]{babel}
\usepackage{amsmath, amsthm, amsfonts, amssymb, accents}
\usepackage{graphicx,dsfont,color}

\usepackage{srcltx,ulem}

\numberwithin{equation}{section}
\theoremstyle{plain}
\newtheorem{theorem}{Theorem}[section]
\newtheorem{lemma}{Lemma}[section]
\theoremstyle{definition}

\newtheorem{remark}{Remark}

\numberwithin{equation}{section}

\def\e{\varepsilon}
\def\g{\gamma}
\def\G{\Gamma}
\def\l{\lambda}
\def\p{\partial}
\def\D{\Delta}

\def\E{\mbox{\rm e}}
\def\a{\alpha}
\def\b{\beta}

\def\Odr{O}
\def\H{W_2}

\def\iu{\mathrm{i}}
\def\di{\,d}

\def\Op{\mathcal{H}}

\def\hf{\mathfrak{h}}

\def\cA{\mathcal{A}}
\def\cL{\mathcal{L}}

\def\H{W_2}

\DeclareMathOperator{\spec}{\sigma}

\DeclareMathOperator{\supp}{supp}
\DeclareMathOperator{\RE}{Re}

\begin{document}

\title{Gap opening in two-dimensional periodic systems}
\author{D.I. Borisov$^1$ and P. Exner$^2$}

\date{\empty}

\maketitle

\allowdisplaybreaks

\begin{abstract}
We present a new method of gap control in two-dimensional periodic systems with the perturbation consisting of a second-order differential operator and a family of narrow potential `walls' separating the period cells in on direction. We show that under appropriate assumptions one can open gaps around points determined by dispersion curves of the associated `waveguide' system, in general any finite number of them, and to control their widths in terms of the perturbation parameter.  Moreover, a distinctive  feature of those gaps is that their  edge values  are attained by the corresponding band functions  at  internal points of the Brillouin zone.
\end{abstract}

\footnotetext[1]{Institute of Mathematics, Ufa Federal Research Center, Russian Academy of Sciences, Ufa, Russia,
Bashkir State University, Ufa, Russia, and
University of Hradec Kr\'alov\'e,  Hradec Kr\'alov\'e, Czech Republic
\\
Email: borisovdi@yandex.ru
}

\footnotetext[2]{Doppler Institute for Mathematical Physics and Applied Mathematics, Czech
Technical University in Prague, B\v{r}ehov\'{a} 7, 11519 Prague, and Nuclear Physics
Institute ASCR, 25068 \v{R}e\v{z} near Prague, Czech Republic
\\
Email: exner@ujf.cas.cz
}

\section{Introduction} \label{s:intro}

Spectral properties of second-order differential operators with periodic coefficients are of considerable interest from more than one reason. On the one hand, it is an interesting mathematical problem with a rich structure. At the same time, such operators are important in description of physical systems, in the first place crystals of various types. To illustrate how involved these problem can be mathematically, it is enough to recall the famous Bethe-Sommerfeld conjecture claiming that in system periodic in more than one direction the number of open spectral gaps is finite \cite{SoBe}. The reasoning that led to it was so natural that the physics community adopted it immediately, however, it took decades to establish its validity in a rigorous mathematical way, cf.~\cite{DaTr, Ka, Pa, PS17, PaSo, Sk79, Sk85, V} and references therein. Nowadays it is done for a large number of systems including a more complicated behavior in some borderline situations \cite{Bo18}.

On the other hand, there are situations, where intuition may mislead you. It was widely believed, for instance, that band edges correspond to quasimomenta laying at the boundary of the corresponding Brillouin zone or at the center of this zone. Was this the case, the task of finding the spectral bands would be easier as the manifold to explore would have one dimension less. It was shown, however, that such a claim does not hold generally in system periodic in more than one direction \cite{HKSW} and this result also extends under additional assumptions to systems periodic in one dimension \cite{EKW}. The corresponding counterexamples used (discrete or metric) quantum graphs rather than Schr\"odinger operators. For those an example of dispersion curves with extrema in the interior of the Brillouin zone was constructed in \cite{BP13b}, while in higher dimensions the question  remained up to now  open.

This is not the only issue  we are going to discuss here. Our goal in  the present paper  is  also  to address  another important problem, the gap control.  Recall that this  question acquired importance recently in connection with the progress in the physics of metamaterials, in other words, artificially prepared periodically structured substances. As the band structure plays decisive role in the 
conductivity properties of the materials, one wants to know whether it is possible to open a gap at a prescribed energy value and to control its width. Various models in which this goal can be achieved have been constructed. In one dimension one can use, for instance, an array of `cells' connected by narrow `windows' \cite{Bo15, Bo16}, 
a  waveguide with a periodically distributed
small windows \cite{BP13a} or a waveguide with a periodic perturbation \cite{JPA13, Na}. An alternative  way proposed  is to place into a waveguide a periodic array of small $\delta'$ traps \cite{EK15}.

In higher dimensions there are fewer results. Khrabustovskyi constructed a model in which gaps can be opened with the help of a lattice of small `pierced resonators' \cite{Kh} and in \cite{EK18} a similar result was obtained by means of a lattice of $\delta'$ traps. The goal of this paper is to present a new method of gap control in a two-dimensional system with a periodic perturbation described by a `small' second order differential operator consisting of raising high and narrow `potential walls' in one direction. We compare this system to the family of parallel waveguides, with the wall replaced by the Dirichlet condition, and show that under appropriate assumption gaps may open around the points where the dispersion curves of the waveguide cross,  in general any finite number of them. Moreover, we are able to control the gap width in terms of the perturbation parameter.  Equally important, the construction answers at the same time the question mentioned above: we  show that the edge values  of the opened gaps are attained by the band functions at internal points of the Brillouin zone.

Let us describe briefly the contents of the paper. In the next section we formulate the problem properly and state our main results as Theorem~\ref{th2.1}. The rest, Sections~\ref{s:approxband} and \ref{proof 2.1} is devoted to the proof.

\section{Formulation of the problem and main result} \label{s:main}

Let $x\in\mathds{R}^2$ be a point expressed through its Cartesian coordinate, $x=(x_1,x_2)$, and let $\Op_0$ be the negative Laplacian in $\mathds{R}^2$. The operator $\Op_0$ is self-adjoint in $L_2(\mathds{R}^2)$ with the domain $\H^2(\mathds{R}^2)$. By $\G$ we denote the rectangular lattice $a_1 \mathds{Z}\times a_2 \mathds{Z}$, where $a_1$, $a_2$ are positive real constants, and $\square:=\{x:\ 0<x_1<a_1,\, 0<x_2<a_2\}$ stands for the corresponding periodicity cell.

Next we introduce the following operator in $L_2(\mathds{R}^2)$,
\begin{equation*}
\cL:=
\frac{\p\ }{\p x_1} A_{11}(x)\frac{\p\  }{\p x_1} + \iu
\Big(A_1(x)\frac{\p\ }{\p x_1}+\frac{\p\ }{\p x_1} A_1(x) \Big) + A_0(x)
\end{equation*}
with the domain $\H^2(\mathds{R}^2)$, where $A_{11}, A_1\in C^1(\mathds{R}^2)$, $A_0\in C(\mathds{R}^2)$ are real functions periodic with respect to the lattice $\G$. The functions $A_{11}$, $A_1$ and $A_0$ are assumed to vanish in the vicinity of the boundary $\p\square$. The operator $\cL$ is obviously symmetric.

By $V$ we denote a real function of the variable $x_2$ satisfying the conditions
\begin{equation}\label{2.2}
\begin{aligned}
&\supp V \subseteq[-2a_3,2a_3], \qquad
V\in C[-2a_3,2a_3],\qquad V\geqslant0,
\\
&  V(x_2)\geqslant c_0>0  \quad\text{in}\quad[- a_3, a_3],
\end{aligned}
\end{equation}
where $a_3$ and $c_0$ are fixed constants. 

The main object of our investigation  is the operator in $L_2(\mathds{R}^2)$ defined as
\begin{equation*}
\Op_\e:=\Op_0 + \e^\a \cL  + \e^{-\frac{3}{2}} V_\e,
\qquad
V_\e=V_\e(x_2):=\sum\limits_{p\in\mathds{Z}} V\left(\frac{x_2-p a_2}{\e}\right),
\end{equation*}
where $\e$ is a small real parameter and $\a\in\big(\tfrac{1}{3},\tfrac{1}{2}\big)$ is a fixed constant. The graphs of the functions $V$ and $V_\e$ are sketched in Figure~1.

It is easy to see that the operator $\Op_\e$ is self-adjoint on the domain $\H^2(\mathds{R}^2)$. Moreover, the operator $\Op_\e$ is periodic with respect to the lattice $\G$, and consequently, its spectrum has a band-gap structure. The main aim of the present work is to study the existence of the spectral gaps in the spectrum of $\Op_\e$ for the parameter small $\e$ enough.

In order to state our main result, we need to introduce additional notations. We begin with the standard formula for the spectrum of the operator $\Op_\e$,
\begin{equation}\label{2.6}
\spec(\Op_\e)=\bigcup\limits_{k\in\mathds{N}} \{E_\e^{(k)}(\tau):\, \tau\in\square^*\},
\end{equation}
where $\tau=(\tau_1,\tau_2)$ is the quasi-momentum and
\begin{equation*}
\square^*:=\left\{\tau:\ |\tau_1|\leqslant \frac{\pi}{a_1},\, |\tau_2|\leqslant \frac{\pi}{a_2}\right\}
\end{equation*}
is the basic cell of the lattice dual to $\Gamma$, or in the physical terminology the Brillouin zone. The symbols $E_\e^{(k)}(\tau)$ stand for the band functions, that is, these are the eigenvalues of the operator $\Op_\e(\tau)$ in $L_2(\square)$ with the differential expression
\begin{equation}\label{2.7}
\Op_\e(\tau)=-\D+\e^\a \cL + \e^{-\frac{3}{2}}V_\e(x)
\end{equation}
subject to the standard quasiperiodic boundary conditions, in other words, the domain of the operator $\Op_\e(\tau)$ consists of the functions $u\in\H^2(\square)$,
which satisfy the boundary conditions
\begin{equation}\label{2.9}
\E^{\iu\tau_j a_j} u\big|_{x_j=0}=u\big|_{x_j=a_j},\quad \E^{\iu\tau_j a_j}\frac{\p u}{\p x_j}\bigg|_{x_j=0}=\frac{\p u}{\p x_j}\bigg|_{x_j=a_j},\quad j=1,2.
\end{equation}
The eigenvalues $E_\e^{(k)}(\tau)$ are supposed to be arranged in the ascending order with the multiplicity taken into account.

\begin{figure}
\includegraphics[scale=0.325]{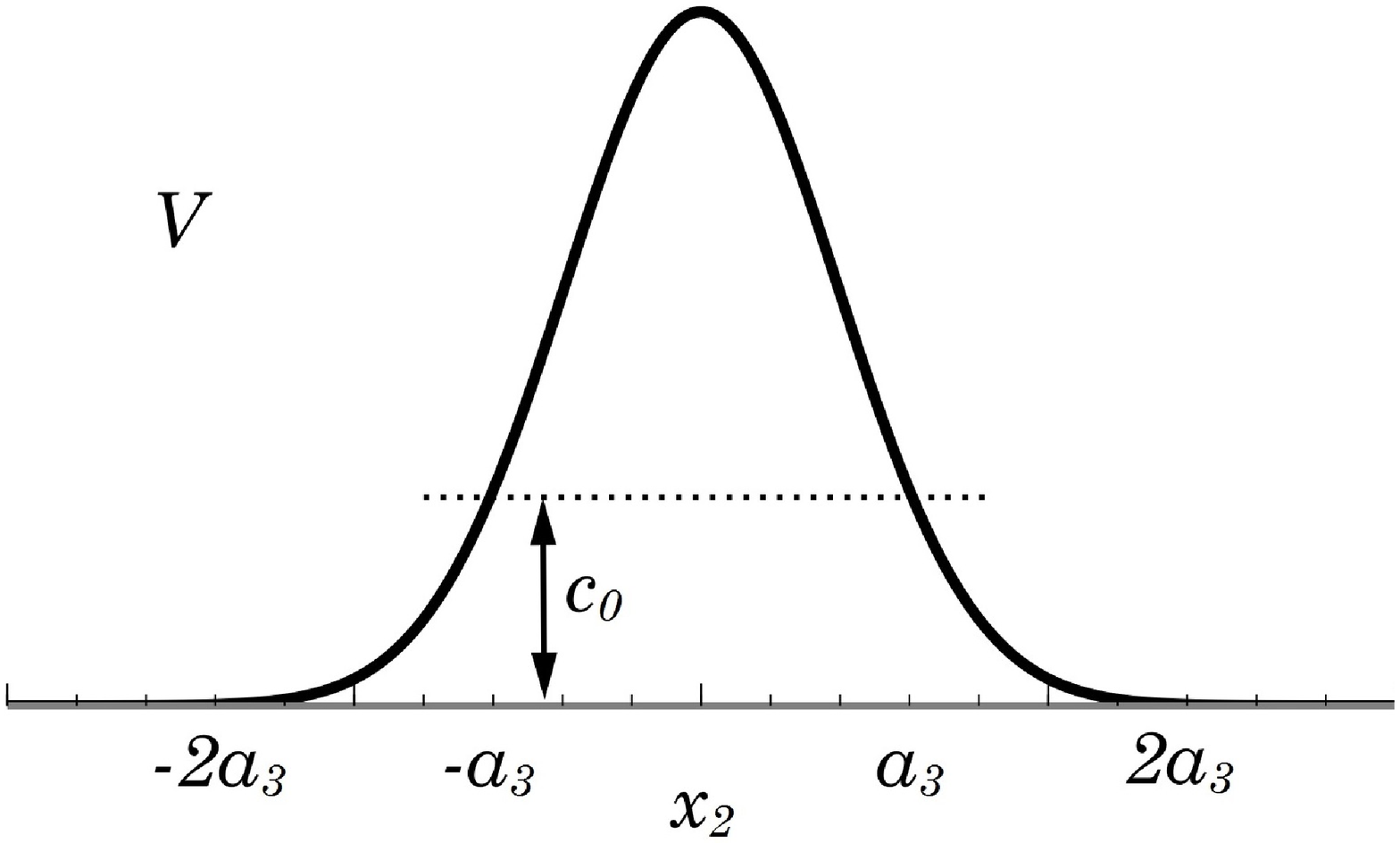}~\hspace{1 true cm}~
\includegraphics[scale=0.325]{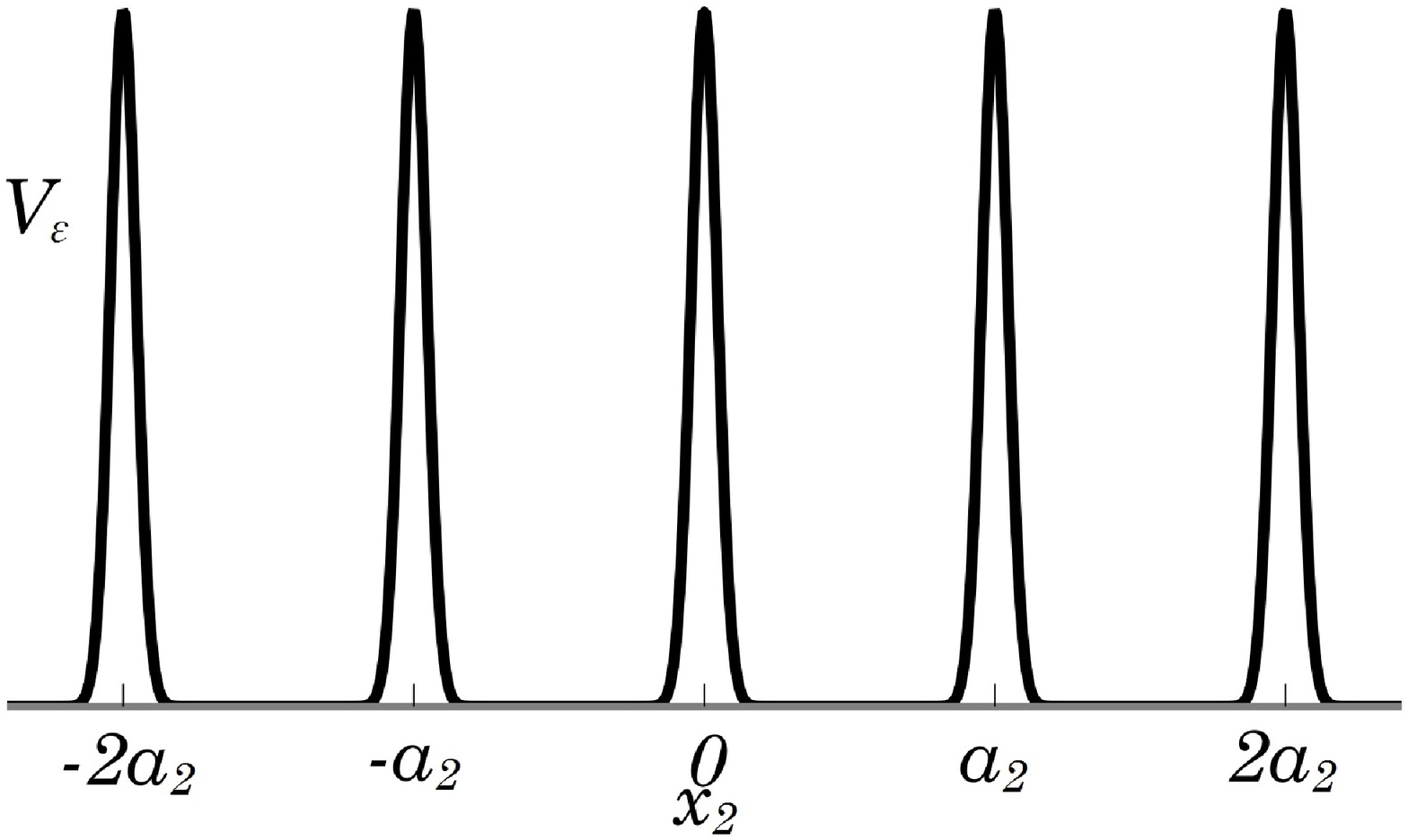}

\caption{\small Sketched graphs of $V$ and $V_\e$}
\end{figure}
We shall show later in the proofs that as $\e\to0+$, the eigenvalues $E_\e^{(k)}(\tau)$ converge to the eigenvalues of the negative Laplacian on $\square$ subject to the quasiperiodic boundary conditions on the lateral boundaries of $\square$ and to the Dirichlet condition as $x_2=0$ and $x_2=a_2$. Up to the exponential factor $\E^{-\iu(\tau_1 x_1+\tau_2 x_2)}$, the eigenfunctions and the corresponding eigenvalues of such limiting operator read as
\begin{align}
&\psi_0^{(n,p)}(x,\tau_2):= \frac{\sqrt{2}}{\sqrt{a_1a_2}}\,
  \E^{\iu \frac{2\pi n}{a_1} x_1}\,\E^{-\iu \tau_2 x_2}\sin\frac{\pi p}{a_2}x_2,\nonumber
\\
&E^{(n,p)}_0(\tau_1):=\left(\tau_1+\frac{2\pi n}{a_1}\right)^2 + \frac{\pi^2 p^2}{a_2^2}.\label{3.22a}
\end{align}
Assume that two such eigenvalues  coincide for some $\tau_1$, namely, let $E_0\in\left(\tfrac{\pi^2}{a_2^2},\tfrac{9\pi^2}{a_2^2}\right)$ be a point such that
\begin{equation}\label{2.10}
E_0=E_0^{(n,1)}(\tau_0)=
\left(\tau_0 + \frac{2\pi n}{a_1}\right)^2 + \frac{\pi^2}{a_2^2} =
E_0^{(m,2)}(\tau_0)=
\left(\tau_0 + \frac{2\pi m}{a_1}\right)^2 + \frac{4\pi^2}{a_2^2}
\end{equation}
for some $\tau_0\in\big[-\tfrac{\pi}{a_1},\tfrac{\pi}{a_1}\big]$ and  some integer $n$, $m$. It is easy to see that triples for which the condition \eqref{2.10} is satisfied always exist, and moreover, we may suppose without loss of generality that at least one of
$n$ and $m$  is nonzero. It is also clear that condition (\ref{2.10}) holds also with $\tau_0$, $n$, $m$ replaced by $-\tau_0$, $-n$, $-m$.
We denote
\begin{align*}
\cL(\tau_1):=&- \Big(\iu \frac{\p\ }{\p x_1}-\tau_1\Big)
A_{11}(x) \Big(\iu \frac{\p\ }{\p x_1}-\tau_1\Big)+ A_1(x) \Big(\iu\frac{\p\ }{\p x_1}-\tau_1\Big)
\\
&+\Big(\iu\frac{\p\ }{\p x_1}-\tau_1\Big) A_1(x)  + A_0(x),
\end{align*}
which is exactly the operator after the Gelfand transform made in the variable $x_1$. The term $\e^\a \cL$ is a regular perturbation in the operator $\Op_\e(\tau)$. Making the regular perturbation theory for the eigenvalues $E_\e^{(k)}(\tau)$ in the vicinity of the point $E_0$, in the further proofs we are led to the matrix
\begin{align*}
M^{(0)}(\tau_1):=&\begin{pmatrix}
M^{(0)}_{11}(\tau_1) & M^{(0)}_{12}(\tau_1)
\\
M^{(0)}_{21}(\tau_1) & M^{(0)}_{22}(\tau_1)
\end{pmatrix}
\\
:=&
\begin{pmatrix}
\big(\psi_0^{(n_*,1)},\cL(\tau_1) \psi_0^{(n_*,1)}\big)_{L_2(\square)} &\big(\psi_0^{(n_*,1)},\cL(\tau_1) \psi_0^{(m_*,2)}\big)_{L_2(\square)}
\\
\big(\psi_0^{(m_*,2)},\cL(\tau_1) \psi_0^{(n_*,1)}\big)_{L_2(\square)} &\big(\psi_0^{(m_*,2)},\cL(\tau_1) \psi_0^{(m_*,2)}\big)_{L_2(\square)}
\end{pmatrix},
\\
n_*:=&\left\{
\begin{aligned}
& n\quad\text{as}\quad \tau_1\geqslant 0,
\\
-&n\quad\text{as}\quad \tau_1<0,
\end{aligned}
\right.\qquad m_*:=\left\{
\begin{aligned}
& m\quad\text{as}\quad \tau_1\geqslant 0,
\\
-&m\quad\text{as}\quad \tau_1<0.
\end{aligned}
\right.
\end{align*}
whose eigenvalues will determine some of leading terms in the asymptotics for $E_\e^{(k)}(\tau)$. The maximal and minimal values of these terms are
\begin{equation}\label{2.16}
\b_l:=\max\{\b_-(\tau_0),\b_-(-\tau_0)\},\qquad \b_r:=\min\{\b_+(\tau_0),\b_+(-\tau_0)\},
\end{equation}
where
\begin{align}
&\begin{aligned}
\b_\pm(\tau_1):=
\pm\frac{|M_{12}^0(\tau_1)|}{|k_3(\tau_1)|}\sqrt{k_3^2(\tau_1)-k_1^2(\tau_1)}-\frac{k_1(\tau_1) k_4(\tau_1)}{k_3(\tau_1)}+k_2(\tau_1),
\label{j2.12}
\end{aligned}
\\
&
\begin{aligned}
k_1(\tau_1)&:=-\frac{2\pi(n_*+m_*)}{a_1}-\tau_1, & k_2(\tau_1)&:=-\frac{M_{11}^{(0)}(\tau_1)+M_{22}^{(0)}(\tau_1)}{2},\\
k_3(\tau_1)&:= \frac{2\pi}{a_1}(n_*-m_*),& k_4(\tau_1)&:=\frac{M_{22}^{(0)}(\tau_1)-M_{11}^{(0)}(\tau_1)}{2}.
\end{aligned}\nonumber
\end{align}
We observe that although the functions $\psi_0^{(n,1)}$ and $\psi_0^{(m,2)}$ depend on $\tau_2$, this is \textit{not} the case for the above introduced functions and matrices. By $\tau_{1,l}$ we denote one of the values $\pm\tau_0$, at which the maximum is attained in the formula for $\b_l$. In the same way, by $\tau_{1,r}$ we denote one of the values $\pm\tau_0$, at which the minimum is attained in the formula for $\b_r$.

Finally, the asymptotics of $E_\e^{(k)}(\tau)$ in the vicinity of $E_0$ involve one more leading term produced by the term $\e^{-\frac{3}{2}}V_\e$ in the operator $\Op_\e(\tau)$. These terms are
\begin{align*}
  &  \l_r := -\frac{8\pi^2}{ a_2^3 \langle V\rangle} \max\big\{|e_r^{(1)}|^2, |e_r^{(2)}|^2\big\},
\quad \l_l := -\frac{8\pi^2}{ a_2^3 \langle V\rangle}
\min\big\{|e_l^{(1)}|^2, |e_l^{(2)}|^2\big\},
\end{align*}
which is well defined because $\langle V\rangle:= \int_{-a_3}^{a_3} v(x)\,\mathrm{d}x\,$ is positive by assumption. To explain the meaning of $e_{l/r}^{(j)}$, we let
\begin{align*}
&t_r=-\frac{k_1(\tau_{1,r}) \big|M_{12}^{(0)}(\tau_{1,r})\big|}{\big|k_3(\tau_{1,r})\big| \sqrt{
k_3^2(\tau_{1,r})-k_1^2(\tau_{1,r})}}-\frac{k_4(\tau_{1,r}) }{k_3(\tau_{1,r})},
\\
&t_l=\frac{k_1(\tau_{1,l}) \big|M_{12}^{(0)}(\tau_{1,l})\big|}{\big|k_3(\tau_{1,l})\big|\sqrt{
k_3^2(\tau_{1,l})
-k_1^2
(\tau_{1,l}) }}-\frac{k_4(\tau_{1,l}) }{k_3(\tau_{1,l})}
\end{align*}
and  we consider the matrices
\begin{equation*}
M^{(0)}(\tau_1)-2t M^{(1)}(\tau_1),\qquad M^{(1)}(\tau_1):=\mathrm{diag}\,\left\{  \frac{2\pi n_*}{a_1}+\tau_1,  \frac{2\pi m_*}{a_1}+\tau_1
\right\},
\end{equation*}
where $(\tau_1,t)$ is either $(\tau_{1,l},t_l)$ or $(\tau_{1,r},t_r)$. Each of these two matrices has two real eigenvalues. By $e_{l}=
\begin{pmatrix}
e_{l}^{(1)}
\\
e_{l}^{(2)}
\end{pmatrix}$ we denote the normalized in $\mathds{R}^2$ eigenvector associated with the smaller eigenvalue of the above matrix with $(\tau_1,t)=(\tau_{1,l},t_l)$, while $e_{r}=
\begin{pmatrix}
e_{r}^{(1)}
\\
e_{r}^{(2)}
\end{pmatrix}$ is the normalized in $\mathds{R}^2$ eigenvector associated with the greater eigenvalue of the above matrix with $(\tau_1,t)=(\tau_{1,r},t_r)$.


Now we are in position to formulate our main result.

\begin{theorem}\label{th2.1}
Assume that the following inequalities hold,
\begin{equation}\label{2.5}
M_{12}^{(0)}(\pm\tau_0)\ne 0,\quad \left(\tau_0 + \frac{2\pi n}{a_1}\right)\left(\tau_0 + \frac{2\pi m}{a_1}\right)<0
 \quad\text{and}\quad
\b_l<\b_r.
\end{equation}
Then there is an $\e_0>0$ such that for all $\e\in(0,\e_0)$ the spectrum of the operator $\Op_\e$ has a gap $\big(\eta_l(\e),\eta_r(\e)\big)$ with the edges that behave asymptotically as
\begin{equation*}
\eta_l(\e)=E_0+\e^\a\b_l+\e^\frac{1}{2}\l_l+\Odr(\e^{2\a}), \quad \eta_r(\e)=E_0+\e^\a\b_r+\e^\frac{1}{2}\l_r
+\Odr\big(\e^{2\a}),
\end{equation*}
and the corresponding band functions  $E_{l/r}(\e,\tau)$ of the operator $\Op_\e$ attain their extrema, i.e. the gap edges $\big(\eta_l(\e), \eta_r(\e)\big)$, at the points  $\tau_{l/r}(\e):=\big(\tau_{l/r}^{(1)}(\e),\tau_{l/r}^{(2)}(\e)\big)$ satisfying the asymptotics
\begin{equation*}
\tau_{l/r}^{(1)}(\e)=\tau_{1,l/r}+\e^\a t_{l/r}+\Odr(\e^{1/2}),\qquad \tau_{l/r}^{(2)}(\e)=\tau_{2,l/r} +\Odr(\e^{1/2}).
\end{equation*}
Here $\tau_{2,r}$ and, respectively, $\tau_{2,l}$, is equal to one of the values $\pm\tfrac{\pi}{a_2}$ or $0$  if $|e_r^{(1)}|>|e_r^{(2)}|$ and, respectively, $|e_l^{(1)}|<|e_l^{(2)}|$ and to one of the values $\pm\tfrac{\pi}{2a_2}$ if $|e_r^{(1)}|<|e_r^{(2)}|$ and, respectively, $|e_l^{(1)}|>|e_l^{(2)}|$.
\end{theorem}

Let us discuss briefly  the meaning of  this theorem. It states that  once the  inequalities (\ref{2.5}) are satisfied, the band spectrum of the operator $\Op_\e$  has a small gap around the point $E_0$ for all  sufficiently small $\e$. The edges of this gap are attained by the values of the  corresponding band functions at points $\tau_l$ and $\tau_r$; let us denote the  coordinates of these points by $\tau_{l/r}^{(j)}$, $j=1,2$. The first  coordinates  $\tau_{l/r}^{(1)}$  are close to   $+\tau_0$ or $-\tau_0$ depending on which of these points is
the minimum and the maximum in (\ref{2.16}). The number $\tau_0$ is defined by the condition (\ref{2.10}) and as we see easily, varying the numbers $a_1$ and $a_2$, we can satisfy the identity (\ref{2.10}) for arbitrary prescribed $\tau_0$ and $E_0$. This means that we can open a gap around a prescribed value $E_0$ so that its edges are attained by the values of the  corresponding band functions at the  points $\tau_{l/r}$ with the first coordinate  being close to the  prescribed $\tau_0$. On the second coordinate $\tau_{l/r}^{(2)}$ we can not say much. The only information we can provide is that the second coordinate is close to one of the values $\{0,\pm \frac{\pi}{a_2}\}$ or $\{\pm\frac{\pi}{2a_2}\}$ if $|e_{l/r}^{(1)}|\ne|e_{l/r}^{(2)}|$. Thus the points $\tau_{l/r}$ are close to one of the following  points: $(\pm\tau_0,\frac{\pi}{a_2})$,  $(\pm\tau_0,-\frac{\pi}{a_2})$, $(\pm\tau_0,0)$, $(\pm\tau_0,\frac{\pi}{2a_2})$,
$(\pm\tau_0,-\frac{\pi}{2a_2})$. In this list, the first two points are located at the boundary of the Brillouin zone, the third one  is located at the middle line $(\tau_1,0)$, while  the two last points are located at the intermediate lines $(\tau_1,\pm\frac{\pi}{2a_2})$. If $|e_{l/r}^{(1)}|=|e_{l/r}^{(2)}|$, our analysis provides no information about localization of $\tau_{l/r}^{(2)}$. In our opinion, this degenerate case could hide a situation, when the second coordinate is close to some value different from those listed above.

We should also stress that if there are several triples $(m,n,\tau_0)$ obeying (\ref{2.10}) with different values $E_0$, and for each of them there exists a corresponding gap in the spectrum of $\Op_\e$ with the above described properties. In particular, this implies that choosing $a_1$ large enough, we can open {\it arbitrarily many gaps} in the spectrum of the operator $\Op_\e$.

Let us we describe briefly the main ideas in the proof of Theorem~\ref{th2.1}; this will clarify also our choice of the powers of $\e$ in the definition of the operator $\Op_\e$. At the first step we show that as $\e\to0+$, the term $\e^{-\frac{3}{2}}V_\e$ can be approximately replaced by the Dirichlet condition at the lines $x_2=p a_2$, while the term $\e^\a\cL$ can be omitted and up to a controlled error, the band functions $E_\e^{(k)}$ are approximated by $E_0^{(n,p)}$. A general picture of location of $E_0^{(n,p)}$ is shown in Figure~2. To understand then the spectral picture in the vicinity of the point $E_0$, we construct the asymptotics for $E_\e^{(k)}$  and  analyse their dependence on $(\tau_1,\tau_2)$. We find that their leading terms are $C_1(\tau_1) \e^\a +C_2(\tau_1,\tau_2)\e^{1/2}$, where $C_1$, $C_2$ are some explicitly calculated functions. The former term is the leading one thanks to our choice of $\a$  and is produced by the term $\e^\a \cL$ in the operator $\Op_\e$. Exactly this term determines the presence of the gaps. The mechanism of opening the gaps is the same as in \cite{JPA13} and is demonstrated in Figure~3. The second term, $C_2\e^{1/2}$, is generated by the term $\e^{-3/2}V_\e$ in the operator $\Op_\e$, and this is a next-to-leading term. This term reflects how the perturbed band functions depend on $\tau_2$ but since this term is smaller than $C_1\e^\a$, it can not destroy above described gap opened due to the leading term. This is also an explanation for our choice of the powers of $\e$ in the operator $\Op_\e$. Namely, the third term could be $\e^{-b} V_\e$ but it can be approximately replaced by the Dirichlet condition only once $b>1$. Then the first term it produces in the asymptotics is $Const\cdot\e^{2-b}$ and it should be smaller than $\e^\a$. To be able to make our analysis, we also need that the error term should be of order $O(\e^{2\a})$. Such situation is realized as $b=\tfrac{3}{2}$ and $\a\in(\tfrac{1}{3},\tfrac{1}{2})$. One more explanation for letting $b=\frac{3}{2}$ is that this choice simplifies some technical details in asymptotic constructions.

\begin{figure}
\centerline{\includegraphics[scale=0.325]{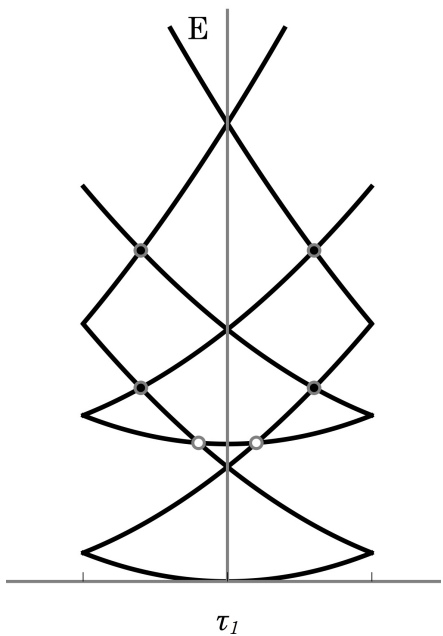}}

\caption{\small Limiting dispersion curves. The intersections of curves indicated by empty circles do not satisfy the second inequality in (\ref{2.5}), while the filled circles indicate the intersections obeying the second inequality in (\ref{2.5}).}
\end{figure}

\section{Approximation of the band functions} \label{s:approxband}

This section is devoted to the preliminary part of the proof of Theorem~\ref{th2.1}. The main idea here is to approximate the band functions $E_{\e}^{(k)}(\tau)$ by similar band functions corresponding to a simpler operator.

By $\Op_0(\tau_1)$ we denote the Laplacian in $\square$ subject to the quasiperiodic boundary conditions of the type \eqref{2.9} on the `longitudinal' boundaries of $\square$, $j=1$, while on
\begin{equation*}
\g:=\{x:\ 0<x_1<a_1,\, x_2=0\}\cup\{x:\ 0<x_1<a_1,\, x_2=a_2\}
\end{equation*}
we impose the Dirichlet condition. Its eigenvalues and the associated eigenfunctions are easily seen to be $E^{(n,p)}_0(\tau_1)$ and
\begin{equation}\label{3.22b}
\Psi^{(n,p)}_0(x,\tau_1)=
\frac{\sqrt{2}}{\sqrt{a_1 a_2}}\, \E^{\iu \left(\tau_1+\frac{2\pi n}{a_1}\right) x_1} \sin\frac{\pi p}{a_2} x_2.
\end{equation}
The main statement we are going to prove in this section is as follows.
\begin{lemma}\label{lm3.3}
Given any fixed $E$, there exists $\e_0>0$ such that for all $\e\in(0,\e_0]$ and  all $(n,p)$ obeying
\begin{equation}
E_0^{(n,p)}(\tau_1)\leqslant E \label{3.24}
\end{equation}
the estimates
\begin{align}\label{3.51}
&\big|E_\e^{(k)}(\tau)-E_0^{(n,p)}(\tau_1)
\big|\leqslant C \e^\a,
\end{align}
hold, where $E_\e^{(k)}(\tau)$ are the band functions appearing in (\ref{2.6}) and $C$ is a constant independent of $\e$, $\mu$, $n$, $p$ and $\tau$ but dependent on $E$. The eigenvalues $E_0^{(n,p)}$ in (\ref{3.51}) are assumed to be arranged in the ascending order counting the multiplicities and this 
establishes the correspondence between the superscripts $k$ and $(n,p)$ in (\ref{3.51}).
\end{lemma}
The proof of this lemma consists of several steps; let us briefly describe them.

We begin with introducing one more operator in
$L_2(\square)$,
which we denote as $\Op_V(\tau)$. The differential expression for $\Op_V(\tau)$ is given by formula (\ref{2.7}) with $\cL=0$ and the boundary conditions are quasiperiodic ones. Since the potential $V$ depends on $x_2$ only, we can find the eigenvalues and the associated eigenfunctions of the operator $\Op_V(\tau)$ by separation of variables,
\begin{equation}\label{3.20}
\begin{aligned}
&E^{(n,p)}_V(\tau)=\left(\tau_1+\frac{2\pi n}{a_1}\right)^2 + \l_\e^{(p)}(\tau_2),
\\
& \Psi^{(n,p)}_V(x,\tau)=
\frac{1}{\sqrt{a_1}}\,\E^{\iu \left(\tau_1+\frac{2\pi n}{a_1}\right) x_1}\, \Psi_\e^{(p)}(x_2,\tau_2),
\end{aligned}
\end{equation}
where $\l_\e^{(p)}$ and $\Psi_\e^{(p)}$ are the eigenvalues and the associated eigenfunctions of the operator
\begin{equation*}
\cA_\e(\tau_2)=-\frac{d^2\ }{dx_2^2} + \e^{-\frac{3}{2}} V_\e \quad \text{in}\;\; L_2(0,a_2)
\end{equation*}
subject to the quasiperiodic boundary conditions. The eigenvalues $\l_\e^{(p)}$ are taken in the ascending order counting the multiplicities  and the associated eigenfunctions are chosen being normalized in $L_2(0,a_2)$.

\begin{figure}
\includegraphics[scale=0.308]{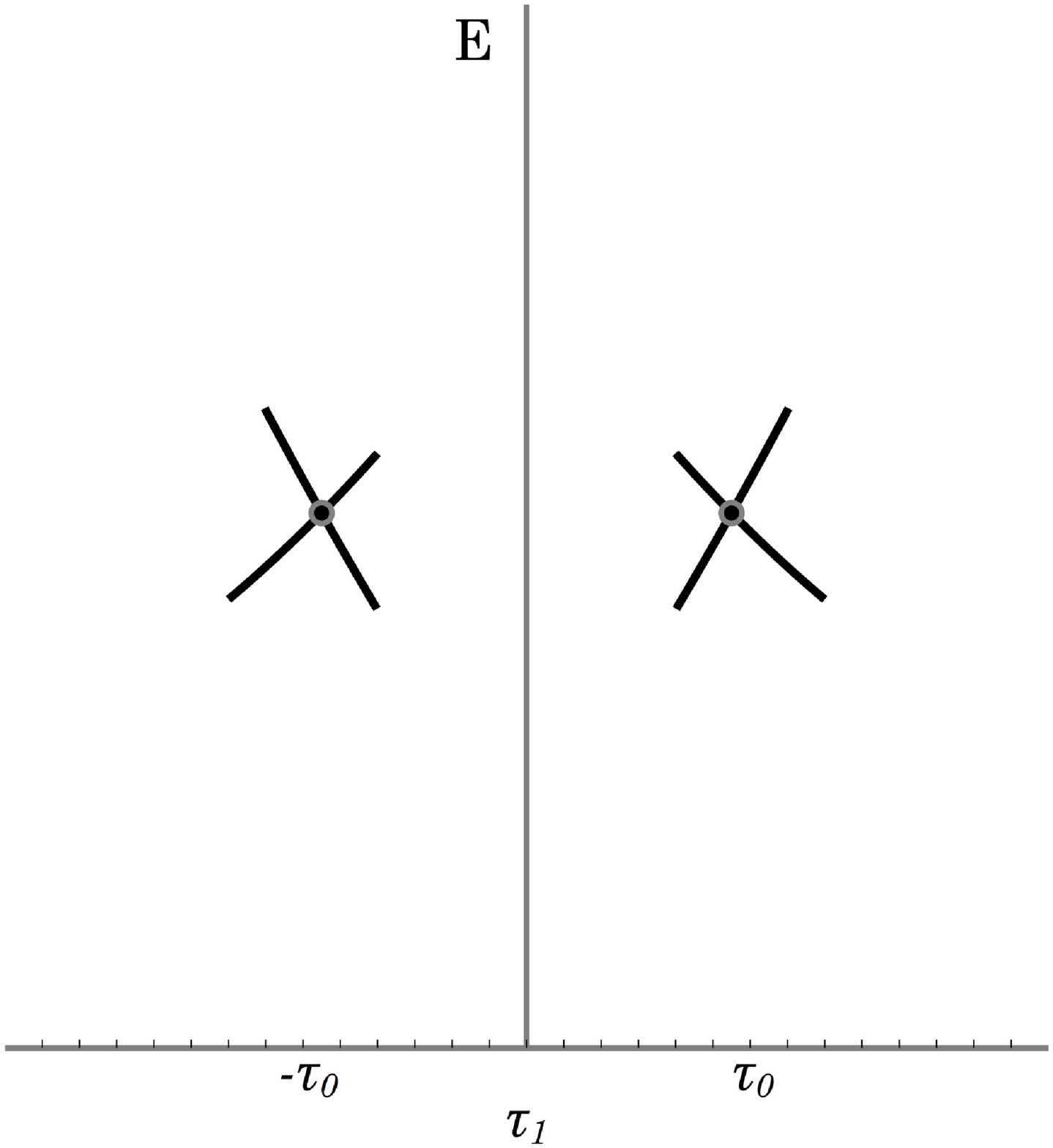}~\hspace{1 true cm}~
\includegraphics[scale=0.33]{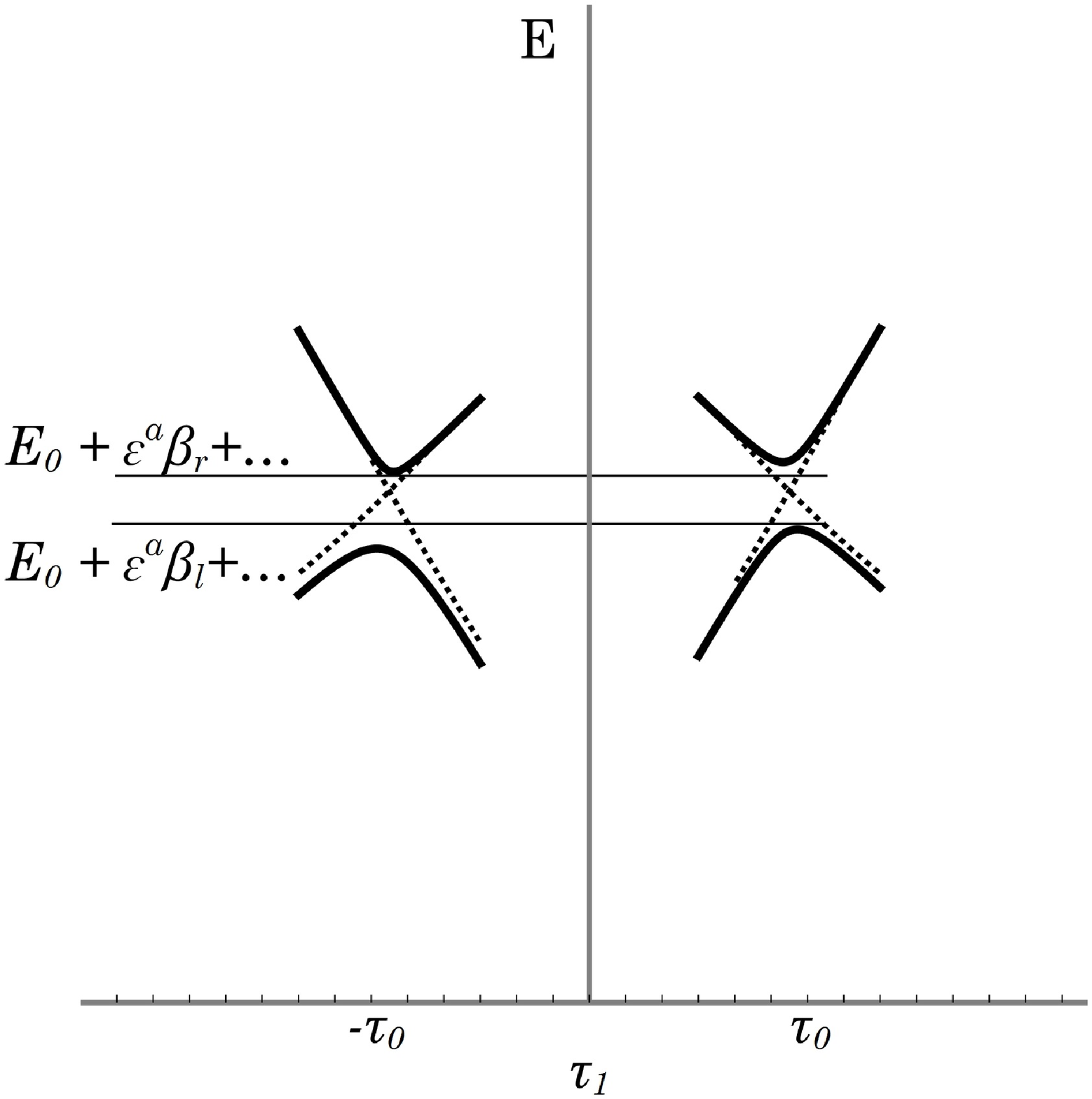}

\caption{\small Under the perturbation $\e^\a\cL$, the intersection of limiting band functions shown in the left figure creates a spectral gap shown in the right figure}
\end{figure}

The first step in the proof of Lemma~\ref{lm3.3} is to estimate the differences between the eigenvalues $E_V^{(n,p)}(\tau)$ and $E_0^{(n,p)}$. A precise formulation of the sought result is formulated in the following lemma.

\begin{lemma}\label{lm3.1}
Given any fixed $E$, there is an $\e_0>0$ such that for all $\e\leqslant \e_0$ and all all $(n,p)$ obeying (\ref{3.24}) the estimates
\begin{equation}
\label{3.25}
\big|E_V^{(n,p)}(\tau)-E_0^{(n,p)}(\tau_1)
\big|\leqslant C \e^\frac{1}{2},
\end{equation}
hold, where $C$ is a constant independent of $\e$, $\mu$, $n$, $p$ and $\tau$ but dependent on $E$.
\end{lemma}

\noindent At the second step, we estimate the differences between the band functions $E_\e^{(k)}(\tau)$ and $E_V^{(n,p)}$ and this will lead us to the statement of Lemma~\ref{lm3.3}.

The proof of Lemma~\ref{lm3.1} consists again of two parts, which we present below as separate Subsections~\ref{ssResAe},~\ref{ssEvsAe}. The second part of the proof of Lemma~\ref{lm3.3} will then follow in Subsection~\ref{ssEvsHe}.

\subsection{Approximation of the resolvent of $\cA_\e$}\label{ssResAe}

In this section we show that as $\e\to0+$, the potential $\e^{-\frac{{3}}{2}}V_\e$ in the operator $\cA_\e$ is replaced by the Dirichlet condition at $x_2=0$ and $x_2=a_2$. Namely,
consider the operator
\begin{equation*}
\cA_0:=-\frac{d^2\ }{dx_2^2} \quad\text{in}\quad L_2(0,a_2)
\end{equation*}
subject to Dirichlet condition at the interval endpoints. We are going to prove that 
as $\e\to0+$, $\cA_\e$ converges to $\cA_0$ in the norm resolvent sense and to estimate the corresponding convergence rate.

Given an $f\in L_2(0,a_2)$, we denote $u_\e:=\big(\cA_\e(\tau_2)-\iu\big)^{-1}f$, $u_0:=(\cA_0-\iu)^{-1}f$, and $v_\e:=u_\e-u_0$. We write the integral identities corresponding to the boundary value problems for $u_\e$ and $u_0$ choosing $v_\e$ as a test function:
\begin{align}\label{3.5}
&
(
u_\e',
v_\e')_{L_2(0,a_2)}+\e^{-\frac{3}{2}}(V_\e u_\e,v_\e)_{L_2(0,a_2)} - \iu(u_\e,v_\e)_{L_2(0,a_2)}=(f,v_\e)_{L_2(0,a_2)},
  \\
&
\begin{aligned}
(
u_0',
v_\e')_{L_2(0,a_2)}
  &-  \overline{u_\e(0)} \left(\E^{\iu\tau_2 a_2}u'_0(a_2)-u'_0(0)\right)
\\
&- \iu(u_\e,v_\e)_{L_2(0,a_2)}=(f,v_\e)_{L_2(0,a_2)}.
\end{aligned}\nonumber
\end{align}
Subtracting these relations from the each other, we get
\begin{align*}
\|
v_\e'\|^2_{L_2(0,a_2)} -&\iu\|v_\e\|_{L_2(0,a_2)}^2 +\e^{-\frac{3}{2}} (V_\e v_\e,v_\e)_{L_2(0,a_2)}
\\
=&-\e^{-\frac{3}{2}} (V_\e u_0,v_\e)_{L_2(0,a_2)} -    \overline{u_\e(0)} \left(\E^{\iu\tau_2 a_2}u'_0(a_2)-u'_0(0)\right),
\end{align*}
an consequently, by virtue of conditions (\ref{2.2}) we obtain
\begin{align}\label{3.7}
&
\|v_\e\|_{\H^1(0,a_2)}^2
\leqslant 2C \e^{-\frac{3}{2}} \|u_0\|_{L_2(S_\e)} \|v_\e\|_{L_2(S_\e)} +2|u_\e(0)| \left(\left|u'_0(0)\right| + \left|u'_0(a_2)\right| \right),
\\[.5em]
&S_\e:=\{x:\ 0<x_2<a_3\e\}\cup\{x:\ a_2-a_3\e<x_2<a_2\},\quad C:=\max\limits_{x_2} V(x_2).\nonumber
\end{align}
\noindent To proceed with estimating in (\ref{3.7}), we shall need the following auxiliary result.
\begin{lemma}
\label{lm3.2}
The functions $u_0,\: u_\e$, and $v_\e$ satisfy the inequalities
\begin{align}\label{3.8}
&|u'_0(0)| + |u'_0(a_2)|\leqslant C\|f\|_{L_2(0,a_2)},
\\
&\|u_0\|_{L_2(S_\e)}\leqslant C\e^\frac{3}{2} \|f\|_{L_2(0,a_2)},\label{3.10}
\\
&
\|v_\e\|_{L_2(S_\e)}\leqslant C\e^{\frac{3}{4}}\|f\|_{L_2(0,a_2)},\label{3.11}
\\
&|u_\e(0)|\leqslant C  \e^{ \frac{1}{4}} \|f\|_{L_2(0,a_2)}\label{3.9}
\end{align}
for all $\e$ small enough, where $C$ is a constant independent of $\e$, $\tau_2$ and $f$.
\end{lemma}
\begin{proof}
Throughout the proof the symbol $C$ stands for various positive constants independent of $x$, $\e$, $\tau_2$, $f$, $v_\e$, $u_0$, and $u_\e$ the values of which are not essential. 
Using standard smoothness improving theorems and the embedding of $\H^1(0,a_2)$ into $C[0,a_2]$, we infer that
\begin{equation}\label{3.12}
|u'_0(0)| + |u'_0(a_2)| \leqslant  C \|u_0\|_{\H^2(0,a_2)} \leqslant C\|u_0\|_{\H^2(0,a_2)} \leqslant C\|f\|_{L_2(0,a_2)},
\end{equation}
which proves (\ref{3.8}). Estimate (\ref{3.10}) is implied by Lemma~4.1 in \cite{PRSE}.


We write the integral identity for $u_\e$ analogous to (\ref{3.5}) employing now this function as the test one,
\begin{equation*}
\|\nabla u_\e\|^2_{L_2(0,a_2)} + \e^{-\frac{3}{2}} (V_\e u_\e,u_\e)_{L_2(0,a_2)} -\iu\|u_\e\|_{L_2(0,a_2)}^2=(f,u_\e)_{L_2(0,a_2)}.
\end{equation*}
In view of condition (\ref{2.2}) this implies
\begin{equation*}
\|u_\e\|_{\H^1(0,a_2)}^2 + c_0\e^{-\frac{3}{2}}\|u_\e\|_{L_2(S_\e)}^2 \leqslant 2\|f\|_{L_2(0,a_2)}\|u_\e\|_{L_2(0,a_2)}.
\end{equation*}
Thus we find
\begin{equation}\label{3.16}
\|u_\e\|_{\H^1(0,a_2)}\leqslant 2\|f\|_{L_2(0,a_2)},\qquad \|u_\e\|_{L_2(S_\e)}\leqslant  2c_0^{-\frac{1}{2}} \e^{\frac{3}{4}} \|f\|_{L_2(0,a_2)}.
\end{equation}
The second estimate in combination with (\ref{3.10}) and the definition of $v_\e$ yields (\ref{3.11}).

It remains to prove (\ref{3.9}). To this aim,
we use integration by parts to rewrite the norm in question as follows,
\begin{equation*}
\|u_\e\|_{L_2(S_\e^+)}^2= 
 \int\limits_{0}^{a_3\e} |u_\e|^2 \di x_2
=  a_3\e|u_\e(0)|^2+2\int\limits_{0}^{a_3\e} (a_3\e-x_2) \RE \overline{u_\e}u'_\e\di x_2.
\end{equation*}
Using next 
(\ref{3.16}) 
and Cauchy-Schwarz inequality
 we arrive at the bound
\begin{align*}
a_3\e|u_\e(0)|^2 \leqslant  \|u_\e\|_{L_2(S_\e^+)}^2 + 2 a_3\e \|u_\e\|_{L_2(S_\e^+)} \|u'_\e\|_{L_2(S_\e^+)}
\leqslant C \e^\frac{3}{2}
\|f\|_{L_2(0,a_2)}^2
\end{align*}
which implies (\ref{3.9}). 
\end{proof}

The proven lemma allows us to proceed with the estimate in (\ref{3.7}) arriving thus at the inequality
$$
\|v_\e\|_{\H^1(0,a_2)}^2  \leqslant C\e^{\frac{1}{4}}  \|f\|_{L_2(0,a_2)}^2,
$$
which is in view of the definition of $v_\e$ equivalent to the estimate
\begin{equation}\label{3.13}
\left\|\big(\cA_\e(\tau_2)-\iu\big)^{-1}-
(\cA_0-\iu)^{-1}
\right\|_{L_2(0,a_2)\to \H^1(0,a_2)}\leqslant C\e^{\frac{1}{8}}.
\end{equation}
Here $\|\cdot\|_{L_2(0,a_2)\to \H^1(0,a_2)}$ denotes the norm of an operator acting from $L_2(0,a_2)$ to $\H^1(0,a_2)$ and $C$ is a constant independent of $\e$.

\subsection{Approximation of the eigenvalues of $\cA_\e$}\label{ssEvsAe}

Let us now complete the proof of Lemma~\ref{lm3.1}. The estimate (\ref{3.13}) and formula (\ref{3.20}) for $E_V^{(n,p)}$ imply immediately that the eigenvalues $E_V^{(n,p)}$ converge to $E_0^{(n,p)}$. In particular, given $E>0$, there exists $\e_0>0$ such that for all  $\e\in(0, \e_0]$  all eigenvalues $E_V^{(n,p)}(\tau)$ with  $(n,p)$ such that $E_0^{(n,p)}(\tau)$ does not obey  (\ref{3.24}) satisfy the lower bound
\begin{equation*}
E_V^{(n,p)}(\tau)>\frac{E}{2}.
\end{equation*}
Consider next the eigenvalues $E_0^{(n,p)}(\tau)$ that do obey (\ref{3.24}). There are finitely many  of them and our aim is to
prove estimate (\ref{3.25})  for each  pair $(n,p)$ for which the inequality  (\ref{3.24}) is satisfied.

The main idea  to achieve this goal  is to construct the asymptotic expansions of $\l_\e^{(p)}$ as $\e\to0+$;
we recall that $\l_\e{(p)}$ was introduced in (\ref{3.20}).
Here we employ a standard approach consisting of two steps. At the first step, we construct the asymptotic expansions formally using the method of matching asymptotic expansions \cite{Il}.  The next step consists of justifying  the formal asymptotics by estimating the error term. We fix $n$
and  adopt  the following  Ansatz  for $\l_\e^{(p)}$,
\begin{equation}\label{3.26}
\l_\e^{(p)}(\tau_2)=\l_0^{(p)} + \e^{\frac{1}{2}} \l_{\frac{1}{2}}^{(p)}(\tau_2) + \e \l_1^{(p)}(\tau_2)   + \ldots,\qquad \l_0^{(p)}:=\frac{\pi^2 p^2}{a_2^2}.
\end{equation}
The asymptotics for the associated eigenfunction $\Psi_\e^{(p)}(x_2,\tau_2)$ is constructed as a combination of
outer and inner expansions. The former reads as
\begin{align}\label{3.27}
&
\Psi_{\e,\mathrm{ex}}^{(p)}(x_2,\tau_2)= \Psi_0^{(p)}(x_2) + \e^{\frac{1}{2}} \Psi_{\frac{1}{2}}^{(p)}(x_2,\tau_2)  + \e \Psi_1^{(p)}(x_2,\tau_2)
 + \ldots,
\\
& \Psi_0^{(p)}(x_2):=\frac{\sqrt{2}}{\sqrt{a_2}} \sin\frac{\pi p x_2}{a_2},\nonumber
\end{align}
while the
inner expansion is constructed as follows,
\begin{align}\label{3.28}
&
\Psi_{\e,\mathrm{in}}^{(p)}(\xi,x_2,\tau_2)=e^{\iu\tau_2 x_2} \Big(\e^{\frac{1}{2}} \Psi_{\frac{1}{2},\mathrm{in}}^{(p)}(\xi,\tau_2) + \e\Psi_{1,\mathrm{in}}^{(p)}(\xi,\tau_2)
 +\ldots\Big),
\\
& \xi=\frac{x_2}{\e}\;\; \text{for}\;\; 0<x_2<2\e^{\frac{1}{2}} \quad \text{and} \quad \xi=\frac{x_2-a_2}{\e}\;\;\text{for}\;\; a_2-2\e^{\frac{1}{2}}<x_2<a_2. \nonumber
\end{align}
 To be explicit, the
 outer expansion is employed to approximate the eigenfunction outside small neighborhoods of the points $x_2=0$ and $x_2=a_2$, while the internal  refers to the behavior  in the vicinity of the mentioned points.
The form of  the Ansatz  (\ref{3.28}) ensures that  the quasiperiodic conditions  are satisfied.

The final approximation for the eigenfunction $\Psi_\e^{(p)}$ is obtained by matching  the two expansions as follows,
\begin{equation}\label{3.50}
\begin{aligned}
\Psi_\e^{(p)}(x_2,\tau_2)= &\Psi_{\e,\mathrm{ex}}^{(p)}(x_2,\tau_2)
\chi\big(x_2\e^{-\frac{1}{2}}\big) \chi\big((a_2-x_2)\e^{-\frac{1}{2}}\big)
 \\
 &+ \Psi_{\e,\mathrm{in}}^{(p)} \big(x_2\e^{-1},\tau_2\big) \Big(1-\chi\big(x_2\e^{-\frac{1}{2}}\big)\Big)
 \\
 &+ \Psi_{\e,\mathrm{in}}^{(p)} \big((x_2-a_2)\e^{-1},\tau_2\big) \Big(1-\chi\big((a_2-x_2)\e^{-\frac{1}{2}}\big)\Big),
\end{aligned}
\end{equation}
where $\chi=\chi(t)$ is an infinitely differentiable function that is equal to one for $t>2$ and vanishes for $t<1$. We substitute the expansions 
(\ref{3.26}), (\ref{3.33}), (\ref{3.28})
into the eigenvalue equation
\begin{equation*}
\left(-\frac{d^2\ }{dx_2^2}+\e^{-\frac{3}{2}}V_\e\right)\Psi_\e^{(p)} = \l_\e^{(p)}\Psi_\e^{(p)}
\end{equation*}
and  identify  the coefficients at the  same  powers of $\e$. This yields the equations
\begin{equation}\label{3.30}
-\frac{d^2\Psi_{\b}^{(p)}}{dx_2^2}=\l_0^{(p)}\Psi_{\b}^{(p)} +\l_{\b}^{(p)}\Psi_0^{(p)}\quad\text{in}\quad(0,a_2),\qquad \b\in\left\{\tfrac{1}{2},1 \right\},
\end{equation}
and
\begin{align}\label{3.31}
&-\frac{d^2\Psi_{\frac{1}{2},\mathrm{in}}^{(p)}}{d\xi^2}=0 \qquad\qquad\qquad\quad
\hphantom{n}\text{in}\;\:\mathds{R},
\\
&-\frac{d^2\Psi_{1,\mathrm{in}}^{(p)}}{d\xi^2}+V(\xi)\Psi_{\frac{1}{2},\mathrm{in}}^{(p)}=0 \qquad\hphantom{i}\text{in}\;\;\mathds{R}.
\label{3.32}
\end{align}
The expansions (\ref{3.27}), (\ref{3.28}) are to be matched in the intermediate zones, namely for $\e^{\frac{1}{2}}<x_2<2\e^\frac{1}{2}$ and $a_2-2\e^\frac{1}{2}<x_2<a_2-\e^\frac{1}{2}$. The asymptotic behavior of the external expansion as $x_2\to 0+$ and $x_2\to a_2-$ should coincide there   with the asymptotic behavior of the internal expansion as $\xi\to\pm\infty$. To be more precise, we match in this way the expansion $e^{-\iu\tau_2 x_2}\Psi_{\e,\mathrm{ex}}^{(p)}$ and $\Psi_{\e,\mathrm{in}}^{(p)}$. In view of the definition of $\Psi_0^{(p)}$ this yields
\begin{align}
&
\begin{aligned}
&\Psi_0^{(p)}(x_2)=\frac{\sqrt{2}\pi p}{a_2^\frac{3}{2}}x_2 +\mathcal{O}(x_2^3), && \quad x_2\to 0+,
\\
&\Psi_0^{(p)}(x_2)=\frac{(-1)^m\sqrt{2}\pi p}{a_2^\frac{3}{2}}(x_2-a_2) +\mathcal{O}\big((a_2-x_2)^3\big), && \quad x_2\to a_2-,
\end{aligned}\nonumber
\\
&
\begin{aligned}
&\E^{-\iu \tau_2 x_2}=1-\iu\tau_2 x_2 +\mathcal{O}(x_2^2), && \hphantom{_1} x_2\to 0+,
\\
&\E^{-\iu \tau_2 x_2}=e^{-\iu\tau_2 a_2}\Big(1-\iu\tau_2 (x_2-a_2) +\mathcal{O}\big((x_2-a_2)^2\big)\Big), &&\hphantom{_1} x_2\to 0+,
\end{aligned}\nonumber
\\
&
\begin{aligned}
&\Psi_\b^{(p)}(x_2)=\Psi_\b^{(p)}(0)+\mathcal{O}(x_2), && \quad x_2\to 0+,
\\
&\Psi_\b^{(p)}(x_2)=\Psi_\b^{(p)}(a_2)+\mathcal{O}(a_2-x_2), &&\quad x_2\to a_2-,
\end{aligned} \label{3.33}
\end{align}
where $\b\in\left\{\tfrac{1}{2},1 \right\}$. The matching conditions are
\begin{align}
&
\begin{aligned}
&\Psi_{\frac{1}{2},\mathrm{in}}^{(p)}(\xi)=\Psi_{\frac{1}{2}}^{(p)}(0)+o(1), && \xi\to+\infty,
\\
&\Psi_{\frac{1}{2},\mathrm{in}}^{(p)}(\xi)=\Psi_{\frac{1}{2}}^{(p)}(a_2)e^{-\iu\tau_2 a_2}+o(1), \qquad\qquad\qquad\qquad \hphantom{\xi\to}&& \xi\to-\infty,
\end{aligned}\label{3.34}
\\
&
\begin{aligned}
&\Psi_{1,\mathrm{in}}^{(p)}(\xi)=\frac{\sqrt{2}\pi p}{a_2^\frac{3}{2}}\xi +
\Psi_1^{(p)}(0)+o(1), &&\quad
\xi\to+\infty,
\\
&\Psi_{1,\mathrm{in}}^{(p)}(\xi)=\left(\frac{(-1)^m\sqrt{2}\pi p}{a_2^\frac{3}{2}}\xi +
\Psi_1^{(p)}(a_2)\right)e^{-\iu\tau_2 a_2}+o(1), &&\quad \xi\to-\infty.
\end{aligned}\label{3.35}
\end{align}
The only solution to the equation (\ref{3.31}) satisfying (\ref{3.34}) is a constant, that is,
\begin{equation}\label{3.36}
\Psi_{\frac{1}{2},\mathrm{in}}^{(p)}(\xi)\equiv K_{\frac{1}{2}},\qquad K_{\frac{1}{2}}:=\Psi_{\frac{1}{2}}^{(p)}(0)=\E^{-\iu\tau_2 a_2}\Psi_{\frac{1}{2}}^{(p)}(a_2),
\end{equation}
and the second identity is to be regarded as the solvability condition.  The general solution to the equation (\ref{3.32}) is given by the formula
\begin{equation*}
\Psi_{1,\mathrm{in}}^{(p)}(\xi)=\tilde\Psi_{1,\mathrm{in}}^{(p)}(\xi) +K_1,\quad \tilde\Psi_{1,\mathrm{in}}^{(p)}(\xi):=\frac{K_{\frac{1}{2}}}{2} \int\limits_{\mathds{R}} |\xi-z|V(z)\di z + C_1\xi.
\end{equation*}
In view of the formul{\ae}
\begin{equation}\label{3.31a}
\begin{aligned}
& \int\limits_{\mathds{R}} |\xi-z|V(z)\di z = \pm   \langle V\rangle \xi \pm \langle z V\rangle,\quad \xi\to\pm \infty,
\\
&   \langle V\rangle:=\int\limits_{\mathds{R}} V(z)\di z,\qquad
 \langle z V\rangle:=\int\limits_{\mathds{R}} z V(z)\di z,
\end{aligned}
\end{equation}
and the conditions (\ref{3.35}) we obtain
\begin{equation*}
\frac{K_{\frac{1}{2}}}{2}\langle V\rangle + C_1=\frac{\sqrt{2}\pi p}{a_2^\frac{3}{2}},
\qquad
-\frac{K_{\frac{1}{2}}}{2}\langle V\rangle + C_1=\frac{(-1)^m\sqrt{2}\pi p}{a_2^\frac{3}{2}} e^{-\iu\tau_2 a_2},
\end{equation*}
arriving thus finally at
\begin{align}\label{3.38}
K_{\frac{1}{2}}=\frac{1}{\langle V\rangle} \frac{\sqrt{2}\pi p \big(1-(-1)^m e^{-\iu\tau_2 a_2}\big)}{a_2^\frac{3}{2}},\qquad
C_1= \frac{\pi p \big(1+(-1)^m e^{-\iu\tau_2 a_2}\big)}{\sqrt{2}a_2^\frac{3}{2}}.
\end{align}
The solvability condition of the problem (\ref{3.30}), (\ref{3.36}), (\ref{3.38}) for $\Psi_{\frac{1}{2}}^{(p)}$ is obtained in the standard way: the equation (\ref{3.30}) should be multiplied by $\Psi_0^{(p)}$ and integrated by parts twice over $(0,a_2)$. This gives an expression for $\l_{\frac{1}{2}}^{(p)}$,
\begin{equation*}
\l_{\frac{1}{2}}^{(p)}(\tau_2)=-  \frac{2\pi^2 p^2 \big|1-(-1)^m e^{-\iu\tau_2 a_2}
\big|^2}{a_2^3\langle V\rangle}.
\end{equation*}
The corresponding solution to problem (\ref{3.30}), (\ref{3.36}), (\ref{3.38}) for $\Psi_{\frac{1}{2}}^{(p)}$ reads as
\begin{equation*}
\Psi_{\frac{1}{2}}^{(p)}(x_2,\tau_2)= 
\frac{\l_{\frac{1}{2}}^{(p)}(\tau_2)}{\sqrt{2a_2}\l_0^{(p)}}  \left(\frac{\pi p}{a_2}x_2\cos\frac{\pi p}{a_2}x_2+\frac{1}{2}\sin\frac{\pi p}{a_2}x_2\right)
+ K_{\frac{1}{2}}\cos\frac{\pi p}{a_2} x_2.
\end{equation*}
This solution is orthogonal to $\Psi_0^{(p)}$ in $L_2(0,a_2)$ as a consequence of  the assumed normalization of the perturbed eigenfunction.

To justify the obtained asymptotics 
(\ref{3.26}), (\ref{3.27}), (\ref{3.28}), (\ref{3.50}),
 the standard argument can be used. Namely, in the same way as above we construct sufficiently many terms in the expansions so that the truncated series of (\ref{3.26}), (\ref{3.50}) solve the eigenvalue equation up to an error of order $O(\e^\frac{1}{2})$. Then we apply Vishik-Lyusternik's lemma, see, for instance, \cite[Sect.~III.1.1, Lemma~1.1]{OIS} or \cite[Sect.~9, Lemma~13]{VL}, which this gives the sought asymptotics,
\begin{equation*}
\l_\e^{(p)}(\tau_2)=\l_0^{(p)} 
+\mathcal{O}(\e^\frac{1}{2}),
\end{equation*}
where the error term is uniform in $\tau_2$. The obtained expansions together with the formul{\ae} (\ref{3.20}), (\ref{3.22a}), (\ref{3.22b}) complete the proof of Lemma~\ref{lm3.1}.

\subsection{Approximation of band functions $E_\e^{(k)}$}\label{ssEvsHe}

 The goal of  this subsection
  is to complete the proof of Lemma~\ref{lm3.3}. The key point will be the following estimate,
\begin{equation}\label{3.52}
\|(\Op_\e(\tau)-\iu)^{-1}-(\Op_V(\tau)-\iu)^{-1}\|_{L_2(\square)\to \H^1(\square)}\leqslant C\e^\a,
\end{equation}
where $C$ is a constant independent of $\e$ and $\tau$. Once this inequality  is  established, it infers  a statement similar to Lemma~\ref{lm3.3}. Specifically, the estimate (\ref{3.52}) yields that to any fixed $E$,
there exists an $\e_0>0$ such that for all $\e\leqslant \e_0$ and all  $(n,p)$  obeying (\ref{3.24}) the estimates
\begin{equation}\label{3.53}
 \big|E_\e^{(k)}(\tau)-E_V^{(n,p)}(\tau) \big|\leqslant C \e^\a
\end{equation}
hold, where $C$ is a constant independent of $\e$, $\mu$, $n$, $p$, and $\tau$ but dependent on $E$, where the eigenvalues $E_0^{(n,p)}$ are assumed to be  arranged  in the ascending order counting the multiplicities. Estimate (\ref{3.53}) and Lemma~\ref{lm3.1} prove Lemma~\ref{lm3.3}. The rest of this subsection is thus  devoted to proving the inequality~(\ref{3.52}).

Given an $f\in L_2(\square)$, we denote $w_\e:=(\Op_\e-\iu)^{-1}f$, $w_V:=(\Op_V-\iu)^{-1}f$, and $w:=w_\e-w_V$. We write the integral identities for $w_\e$ and $w_V$ employing $w$ as the test function,
\begin{align*}
&(\nabla w_V,\nabla w)_{L_2(\square)} + \e^{-\frac{3}{2}} (V_\e w_V,w)_{L_2(\square)}-\iu(w_V,w)_{L_2(\square)}=(f,w)_{L_2(\square)},
\\
&(\nabla w_\e,\nabla w)_{L_2(\square)}
 + \e^{-\frac{3}{2}} (V_\e w_\e,w)_{L_2(\square)}+\e^\a\hf(w_\e,w) -\iu(w_\e,w)_{L_2(\square)}=(f,w)_{L_2(\square)},
\\
&\hf(w_\e,w):=\left(A_{11} \frac{\p w_\e}{\p x_j},\frac{\p w}{\p x_i}\right)_{L_2(\square)} + \iu
\left(A_1
\frac{\p w_\e}{\p x_1
},w\right)_{L_2(\square)}
\\
&\hphantom{\hf(w_\e,w):=}+\iu \left(w_\e, A_1\frac{\p w}{\p x_j} \right)_{L_2(\square)}+(A_0 w_\e,w)_{L_2(\square)}.
\end{align*}
 Subtracting  the first two identities from  each other  and employing the representation $\hf(w_\e,w)=\hf(w,w) +\hf(w_V,w)$, we obtain
\begin{equation}\label{3.54}
\|\nabla w\|_{L_2(\square)}^2
 + \e^{-\frac{3}{2}} (V_\e w,w)_{L_2(\square)}+\e^\a\hf(w,w) -\iu\|w\|_{L_2(\square)}^2=-\e^\a \hf(w_V,w).
\end{equation}
 In view of  the  assumed positivity of the function $V_\e$, a similar identity for $w_V$ with the test function $w_V$,
\begin{equation*}
\|\nabla w_V\|_{L_2(\square)}^2
 + \e^{-\frac{3}{2}} (V_\e w_V,w_V)_{L_2(\square)} -\iu\|w\|_{L_2(\square)}^2=(f,w_V)_{L_2(\square)},
\end{equation*}
implies the \textit{apriori} estimate
\begin{equation*}
\|w_V\|_{\H^1(\square)}\leqslant C\|f\|_{L_2(\square)}\,;
\end{equation*}
the symbol $C$ stands  again  for various inessential constants independent of $\e$, $\tau$, and $f$.  This allows us to  estimate the right hand in (\ref{3.54}) as
\begin{equation}\label{3.56}
\e^\a |\hf(w_V,w)|\leqslant C\e^\a \|w_V\|_{\H^1(\square)} \|w\|_{\H^1(\square)}.
\end{equation}
The real part of left-hand side in (\ref{3.54}) can be estimated from below as
\begin{align*}
\|\nabla w\|_{L_2(\square)}^2
 + \e^{-\frac{3}{2}} (V_\e w,w)_{L_2(\square)}+\e^\a\hf(w,w) \geqslant &\|\nabla w\|_{L_2(\square)}^2
 +\e^\a\hf(w,w)
 \\
  \geqslant &\frac{1}{2}  \|\nabla w\|_{L_2(\square)}^2
 -C\e^\a\|w\|_{L_2(\square)}^2.
\end{align*}
Using this result, (\ref{3.56}), and taking the real and the imaginary part of (\ref{3.54}), we get
\begin{align*}
&
\|w\|_{L_2(\square)}^2\leqslant C\e^\a \|f\|_{L_2(\square)} \|w\|_{\H^1(\square)},
\\
&\frac{1}{2}\|\nabla w\|_{L_2(\square)}^2 -C\e^\a \|w\|_{L_2(\square)}^2\leqslant   C\e^\a \|f\|_{L_2(\square)} \|w\|_{\H^1(\square)},
\end{align*}
and consequently,
\begin{equation*}
\|w\|_{\H^1(\square)}\leqslant C\e^\a\|f\|_{L_2(\square)}
\end{equation*}
which finally proves (\ref{3.52}).

\section{Proof of Theorem~\ref{th2.1}} \label{proof 2.1}

Now we are in position to prove our main result, Theorem~\ref{th2.1}. In view of Lemma~\ref{lm3.3}, the band functions $E_\e^{(k)}(\tau)$ converge to the eigenvalues $E_0^{(n,p)}$ as $\e\to 0+$. Considering a small fixed neighbourhood of the point $E_0$, we immediately conclude from Lemma~\ref{lm3.3} that there are constants $C_1>0$, $C_2>0$ such that as $|\tau_1-\tau_0|<C_1\e^\a$, the segment $[E_0-C_2\e^\a,E_0+C_2\e^\a]$ contains no eigenvalues of the operator $\Op_\e(\tau)$ except exactly one pair of them converging to $E_0^{(n,1)}(\tau)$ and $E_0^{(m,2)}(\tau)$, respectively, as $\e\to 0+$.
Let us analyze the behavior of these two eigenvalues.

We first rewrite the eigenvalue equation by changing the eigenfunction $\psi\mapsto \E^{\iu (\tau_1 x_1+\tau_2 x_2)}\psi$. This leads to a new equation
\begin{equation}\label{4.0}
\tilde{\Op}_\e(\tau)\psi=E\psi,
\end{equation}
where $\tilde{\Op}_\e$ is a self-adjoint operator in $L_2(\square)$ with the differential expression
\begin{align*}
\tilde{\Op}_\e(\tau)=&\sum\limits_{j=1}^{2} \left(\iu \frac{\p\ }{\p x_j}-\tau_j\right)^2
-\e^\a\left(\iu \frac{\p\ }{\p x_1}-\tau_1\right) A_{11}(x)\left(\iu \frac{\p\ }{\p x_1}-\tau_1\right)
 \\
&+  \e^\a
\left(A_1(x)\frac{\p\ }{\p x_1}-\tau_1\right)+\left(\iu\frac{\p\ }{\p x_1}-\tau_1\right) A_1(x)+\e^\a A_0(x)+ \e^{-\frac{3}{2}}V_\e(x)
\end{align*}
subject to periodic boundary conditions on the boundary $\p\square$.

Next we introduce an auxiliary parameter $t$ setting $\tau_1=t_0+\e^\a t$ and we denote the perturbed eigenvalues in question by $E_\e^\pm(t)$. The parameter $t$ ranges over some segment $[-C,C]$ for a sufficiently large $C$. The associated eigenfunctions of the operator $\tilde{\Op}_\e$ are denoted by $\Psi_\e^\pm(x,t)$. We are going to construct the first terms in the asymptotic expansions of $E_\e^\pm(t)$. This will be done using the same scheme as in Subsection~\ref{ssEvsAe} taking into consideration the presence of the perturbation $\e^\a \cL$ in the operator.

We adopt the following Ans\"atze for the eigenvalues $E_\e^\pm(t,\tau_2)$,
\begin{equation}\label{4.1}
E_\e^{\pm}(t,\tau_2)=E_0 + \e^\a \l_\a^\pm(t) + \e^\frac{1}{2}\l_1^\pm(t,\tau_2) + \ldots
\end{equation}
The asymptotics for the associated eigenfunctions are again constructed as a combination of the
outer and inner
expansions. The former is introduced as
\begin{equation}\label{4.2}
\Psi_{\e,\mathrm{ex}}^\pm(x,t,\tau_2)=
\Psi_0^\pm(x,t,\tau_2)+\e^\a \Psi_\a^\pm(x,t,\tau_2) +\e^\frac{1}{2}\Psi_\frac{1}{2}^\pm(x,t,\tau_2)+\ldots,
\end{equation}
where
\begin{equation*}
\Psi_0^\pm(x,t,\tau_2):=\frac{\sqrt{2}}{\sqrt{a_1a_2}}
\left(c_n^\pm(t,\tau_2) \psi_0^{(n,1)}(x,\tau_2)
+ c_m^\pm(t,\tau_2) 
\psi_0^{(m,2)}(x,\tau_2)\right),
\end{equation*}
and $c_j^\pm(t,\tau_2)$ are constants to be determined. The  inner expansion is of the form
\begin{equation}\label{4.3}
\Psi_{\e,\mathrm{in}}^\pm(x,t,\tau_2)=
\e^\frac{1}{2}\Psi_{\frac{1}{2},\mathrm{in}}^\pm(\xi,x_1,t,\tau_2)
+ \e \Psi_{1,\mathrm{in}}^\pm(\xi,x_1,t,\tau_2) + \ldots,
\end{equation}
where the variable $\xi$ is   the same  as in (\ref{3.28}). The approximation for the eigenfunctions is defined via the external and inner expansion matching as in (\ref{3.50}):
\begin{align*}
\Psi_\e^\pm(x,t,\tau_2)= &\Psi_{\e,\mathrm{ex}}^\pm(x,t,\tau_2)
\chi\big(x_2\e^{-\frac{1}{2}}\big) \chi\big((a_2-x_2)\e^{-\frac{1}{2}}\big)
 \\
 &+ \Psi_{\e,\mathrm{in}}^\pm \big(x_2\e^{-1},x_1,t,\tau_2\big) \Big(1-\chi\big(x_2\e^{-\frac{1}{2}}\big)\Big)
 \\
 &+ \Psi_{\e,\mathrm{in}}^\pm \big((x_2-a_2)\e^{-1},x_1,t,\tau_2\big) \Big(1-\chi\big((a_2-x_2)\e^{-\frac{1}{2}}\big)\Big).
\end{align*}
We substitute the Ans\"atze (\ref{4.1}), (\ref{4.2}), (\ref{4.3}) into the equation (\ref{4.0}) and collect the coefficients at the same powers of $\e$. This gives the equations
\begin{align}\label{4.6}
&\tilde{\Op}_0\Psi_\a^\pm-E_0\Psi_\a^\pm=2t \left(\iu\frac{\p\ }{\p x_1}-\tau_0\right) \Psi_0^\pm -\cL(\tau_0)\Psi_0^\pm + \l_\a^\pm \Psi_0^\pm,
\\
&\tilde{\Op}_0\Psi_\frac{1}{2}^\pm-E_0\Psi_\frac{1}{2}^\pm= \l_\frac{1}{2}^\pm \Psi_0^\pm,\nonumber 
\end{align}
both in $\square$, where the differential expression
\begin{equation*}
\tilde{\Op}_0:=\left(\iu\frac{\p\ }{\p x_1}-\tau_0\right)^2 + \left(\iu\frac{\p\ }{\p x_2}-\tau_2\right)^2
\end{equation*}
has to be amended with boundary conditions. On the lateral boundaries we postulate the periodic ones,
\begin{equation}\label{4.8}
\Psi_\b^\pm\big|_{x_1=0}=\Psi_\b^\pm\big|_{x_1=a_1},\qquad \frac{\p\Psi_\b^\pm}{\p x_1}\bigg|_{x_1=0}= \frac{\p\Psi_\b^\pm}{\p x_1}\bigg|_{x_1=a_1},
\end{equation}
recall that $\b\in\left\{\tfrac{1}{2},1 \right\}$. In contrast, Dirichlet boundary condition
is  assumed for $\Psi_\a^\pm$,
\begin{equation}\label{4.9}
\Psi_\a^\pm=0\quad\text{on}\;\;\g.
\end{equation}
The boundary condition for $\Psi_\frac{1}{2}^\pm$ on $\g$ will be determined later, by matching with the inner expansion.

\begin{remark}\label{rm4.1}
The homogeneous Dirichlet condition for $\Psi_\a^\pm$ have been postulated from the following reason. We could have assumed that this function has some unknown values on $\g$ to be determined by the matching with the inner expansion as it was done in Subsection~\ref{ssEvsAe}. Then one  would have to introduce an additional term $\e^{\frac{1}{2}+\a}\Psi_{\frac{1}{2}+\a}^{(m)}$ in the external expansion and terms $\e^\a\Psi_{\a,\mathrm{in}}^\pm + \e^{\frac{1}{2}+\a}\Psi_{\frac{1}{2}+\a,\mathrm{in}}^\pm$ in the inner expansion. Matching of $\Psi_\a^\pm$ with these extra terms then implies that this function  has to  vanish on $\g$ and all these extra terms are zero.  With this fact in mind we adopt from the beginning the homogeneous Dirichlet condition on $\g$ for $\Psi_\a^\pm$  obtaining in this way a simpler Ansatz  for the perturbed eigenfunctions.
\end{remark}

In view of the condition (\ref{4.9}), the solvability criterion  for the problem (\ref{4.6}), (\ref{4.8}), (\ref{4.9}) reduces to  the orthogonality of the right-hand in the equation to the functions $\psi_0^{(n,1)}$ and $\psi_0^{(m,2)}$ in $L_2(\square)$. These requirements  can be written as the system of linear equations,
\begin{align}
& M(t) C^\pm(t)=\l_\a^\pm(t) C^\pm(t), && M(t):=M^{(0)}(\tau_0)-2t M^{(1)}(\tau_0),\nonumber
\\
&
 C^\pm(t):=
\begin{pmatrix}
c_n^\pm(t)
\\
c_m^\pm(t)
\end{pmatrix}. \nonumber
\end{align}
Hence,
 $\l_\a^\pm$ are the eigenvalues of the matrix $M$ and $C^\pm$ are the associated eigenfunctions. Since the matrix $M$ is Hermitian, we  can choose the vectors $C^\pm$  to be  orthonormal in $\mathds{C}^2$. It is straightforward to find the eigenvalues $\l_\a^\pm$ explicitly,
\begin{align*}
\l_\a^\pm(t):=& \frac{M_{11}^{(0)}(\tau_0) + M_{22}^{(0)}(\tau_0)}{2}  -2t \left(\frac{\pi (n+m)}{a_1}+\tau_0
\right)
\\
&\pm \Bigg( \left( \frac{M_{11}^{(0)}(\tau_0) - M_{22}^{(0)}(\tau_0)}{2}  -t  \frac{\pi (n-m)}{a_1}
\right)^2
+ \big|  M_{12}^{(0)}(\tau_0)
\big|^2
\Bigg)^\frac{1}{2}.
\end{align*}

Let us next  determine the boundary conditions  to be imposed on $\Psi_\frac{1}{2}^\pm$ on  the boundary  $\g$. We substitute the expansions (\ref{4.3}) and (\ref{4.1}) into the eigenvalue equation (\ref{4.1}), pass to the variable $\xi$, and collect the coefficients at the  same  powers of $\e$. This leads us to the equations for the coefficients of the inner expansion,
\begin{align}\label{4.12}
&-\frac{d^2\Psi_{\frac{1}{2},\mathrm{in}}^\pm}{d\xi}=0  \hphantom{+V\Psi_{\frac{1}{2},\mathrm{in}}^\pm}\quad \text{in}\;\;\mathds{R},
\\
&-\frac{d^2\Psi_{1,\mathrm{in}}^\pm}{d\xi}+V\Psi_{\frac{1}{2},\mathrm{in}}^\pm=0 \quad \text{in}\;\;\mathds{R}.\label{4.13}
\end{align}
We expand the functions $\Psi_0^\pm$, $\Psi_\a^\pm$ and $\Psi_\frac{1}{2}^\pm$ as $x_2\to 0+$ and $x_2\to a_2-$,
\begin{align*}
& \Psi_0^\pm(x,t)=\phi_{0,+}^\pm(x_1,t)x_2+\mathcal{O}(x_2^3),\qquad x_2\to 0+,
\\
& \Psi_0^\pm(x,t)=\phi_{0,-}^\pm(x_1,t)(x_2-a_2)+\mathcal{O}\big((x_2-a_2)^3\big),\qquad x_2\to a_2-,
\\
&\phi_{0,+}^\pm(x_1,t):=\frac{\sqrt{2}\pi}{a_1^\frac{1}{2}a_2^\frac{3}{2}} \Big(c_n^\pm(t) e^{\iu\frac{2\pi n}{a_1}x_1} + 2 c_m^\pm(t)\,\E^{\iu\frac{2\pi m}{a_1}x_1}\Big),
\\
&\phi_{0,-}^\pm(x_1,t):=\frac{\sqrt{2}\pi}{a_1^\frac{1}{2}a_2^\frac{3}{2}} \Big(-c_n^\pm(t) e^{\iu\frac{2\pi n}{a_1}x_1} + 2 c_m^\pm(t)\, \E^{\iu\frac{2\pi m}{a_1}x_1}\Big)\,\E^{-\iu\tau_2 a_2},
\\
&\Psi_\a^\pm(x,t,\tau_2)=\mathcal{O}(x_2),\quad x_2\to 0+,
\\
&\Psi_\a^\pm(x,t,\tau_2)=\mathcal{O}(x_2-a_2),\quad x_2\to a_2-,
\\
&\Psi_\b^\pm(x,t,\tau_2)=\Psi_\b^\pm(x_1,0,t,\tau_2)+\mathcal{O}(x_2),\quad x_2\to 0+,
\\
&\Psi_\b^\pm(x,t,\tau_2)=\Psi_\b^\pm(x_1,a_2,t,\tau_2)+\mathcal{O}(x_2-a_2),\quad x_2\to a_2-,\quad \b\in\{\tfrac{1}{2},1\}.
\end{align*}
Rewriting these formulae in the variable $\xi$, by matching condition we conclude that the coefficients of the inner expansion should behave at infinity as follows,
\begin{align}\label{4.14}
&
\begin{aligned}
&
\Psi_{\frac{1}{2},\mathrm{in}}^\pm(\xi,x_1,t,\tau_2)=\Psi_{\frac{1}{2}}^\pm(x_1,0,t,\tau_2) + o(1),\qquad\hphantom{_2} \xi\to+\infty,
\\
&\Psi_{\frac{1}{2},\mathrm{in}}^\pm(\xi,x_1,t,\tau_2)=\Psi_{\frac{1}{2}}^\pm(x_1,a_2,t,\tau_2) + o(1),\qquad \xi\to-\infty,
\end{aligned}
\\
&
\begin{aligned}
&
\Psi_{1,\mathrm{in}}^\pm(\xi,x_1,t,\tau_2)=\phi_{0,+}^\pm(x_1,t)\xi + o(|\xi|),
 \hphantom{(1),\xi,}\quad \xi\to+\infty,
\\
&\Psi_{1,\mathrm{in}}^\pm(\xi,x_1,t,\tau_2)=\phi_{0,-}^\pm(x_1,t)\xi + o(|\xi|), \hphantom{(1),\xi,}\quad \xi\to-\infty.
\end{aligned}\label{4.15}
\end{align}
The problem (\ref{4.12}), (\ref{4.14}) is solvable if and only if
\begin{equation}\label{4.16}
\Psi_{\frac{1}{2}}^\pm(x_1,0,t,\tau_2) = \Psi_{\frac{1}{2}}^\pm(x_1,a_2,t,\tau_2) =:\phi_\frac{1}{2}^\pm(x_1,t,\tau_2)
\end{equation}
and its solution reads as
\begin{equation*}
\Psi_\frac{1}{2}(\xi,x_1,t,\tau_2)=\phi_\frac{1}{2}^\pm(x_1,t,\tau_2).
\end{equation*}
 Next we  proceed to equation (\ref{4.13})  writing its solution as
\begin{equation*}
\Psi_{1,\mathrm{in}}^\pm(\xi,x_1,t,\tau_2)=\frac{\phi_{\frac{1}{2}}(x_1,t,\tau_2)}{2} \int\limits_{\mathds{R}} |\xi-z| V(z)\di z + T_1(x_1,t,\tau) \xi + T_0(x_1,t,\tau_2),
\end{equation*}
where $T_0$, $T_1$ are functions  independent of $\xi$. The behavior of the function $\Psi_{1,\mathrm{in}}^\pm$ at infinity can be expressed using formul{\ae} (\ref{3.31a}); comparing the result with (\ref{4.15}), we get
\begin{equation*}
  \frac{\phi_{\frac{1}{2}}^\pm}{2}\langle V\rangle + T_1=\phi_{0,+}^\pm,
\qquad
-\frac{\phi_{\frac{1}{2}}^\pm}{2}\langle V\rangle + T_1=\phi_{0,-}^\pm,
\end{equation*}
which determines $\phi_\frac{1}{2}^\pm$,
\begin{equation}\label{4.18}
\phi_\frac{1}{2}^\pm(x_1,t,\tau_2)= \frac{\phi_{0,+}^\pm - \phi_{0,-}^\pm}{\langle V\rangle}.
\end{equation}
The solvability of the problem  determined by thr boundary conditions (\ref{4.8}), (\ref{4.9}), (\ref{4.16}), (\ref{4.18}) is obtained in the standard way, that is, the equation (\ref{4.8}) is multiplied by $\psi_0^{(n,1)}$ and $\psi_0^{(m,2)}$ and integrated twice by parts over $\square$ taking  the indicated  conditions into account. This yields $\l_\frac{1}{2}^\pm$  in the form of the following expression
\begin{equation}\label{4.19}
\begin{aligned}
\l_\frac{1}{2}^\pm(t,\tau_2)=&-\frac{2\pi^2}{a_1 a_2^3 \langle V\rangle} \int\limits_{0}^{a_1} \Big| c_n^\pm(t) e^{\iu\frac{2\pi n}{a_1}x_1} (1+\E^{-\iu\tau_2 a_2}) + c_m^\pm(t)\,\E^{\iu\frac{2\pi m}{a_1}x_1} (1-\E^{-\iu\tau_2 a_2})
\Big|^2 dx_1
\\
=& -\frac{8\pi^2}{ a_2^3 \langle V\rangle} \Big(\big(c_n^\pm(t)\big)^2\cos^2\tau_2 a_2 + \big(c_m^\pm(t)\big)^2\sin^2\tau_2 a_2 \Big).
\end{aligned}
\end{equation}
The justification of the asymptotics (\ref{4.1}) can be done in the same way is in Subsection~\ref{ssEvsAe}: we need to construct sufficiently many terms in  the expansion  to get an error of order $\mathcal{O}(\e^{2\a})$ and then we can apply the Vishik-Lyusternik's lemma. Finally, this allows us to conclude that the asymptotic expansions  of the eigenvalues $E_\e^{\pm}(t,\tau_2)$ are
\begin{equation*}
E_\e^{\pm}(t,\tau_2)=E_0 + \e^\a \l_\a^\pm(t) + \e^\frac{1}{2}\l_\frac{1}{2}^\pm(t,\tau_2) + O(\e^{2\a})
\end{equation*}
uniformly in $t$ and $\tau_2$. In order to complete the proof of Theorem~\ref{th2.1}, it is sufficient to calculate the extrema of the leading terms in the above asymptotics and compare them mutually.

Let us inspect the behavior of the said eigenvalues with respect to $t$ and $\tau_2$. By straightforward calculations one can check that the extrema of the functions $\l_\a^\pm$ are given by the formul{\ae}
\begin{align*}
&\min\limits_{\mathds{R}} \l_\a^+(t)=\l_\a^+(t_+),\qquad \max\limits_{\mathds{R}} \l_\a^-(t)=\l_\a^-(t_-),\qquad
\l_\a^\pm(t_\pm)=\b_\pm(\tau_0),\nonumber
\\
&t_\pm=\mp\frac{k_1(\tau_0) \big|M_{12}^{(0)}(\tau_0)\big|}{\big|k_3(\tau_0)\big|\sqrt{
k_3^2(\tau_0)
-k_1^2
(\tau_0) }}-\frac{k_4(\tau_0) }{k_3(\tau_0)},
\end{align*}
where $\b_\pm$ are the functions from (\ref{j2.12}). It follows immediately from the formula (\ref{4.19}) that for each $t$, the extrema of $\l_{\frac{1}{2}}^\pm(t,\tau_2)$ are attained at the points $\tau_2=-\tfrac{\pi}{a_2}$,  $\tau_2=0$, $\tau_2=\tfrac{\pi}{a_2}$ if $|c_n^\pm(t)|>|c_m^\pm(t)|$ and they are attained at the points $\tau_2=\pm\frac{\pi}{2a_2}$ if  $|c_n^\pm(t)|<|c_m^\pm(t)|$. If $|c_n^\pm(t)|=|c_m^\pm(t)|$, the function $\l_{\frac{1}{2}}^\pm(t,\tau_2)$ is independent of $\tau_2$. Finally, we  have
\begin{align*}
&\min\limits_{\tau_2} \l_{\frac{1}{2}}^+(t,\tau_2) = -\frac{8\pi^2}{ a_2^3 \langle V\rangle} \max\big\{(c_n^+(t)\big)^2, \big(c_m^+(t)\big)^2\big\},
\\
&\max\limits_{\tau_2} \l_{\frac{1}{2}}^-(t,\tau_2) = -\frac{8\pi^2}{ a_2^3 \langle V\rangle} \min\big\{(c_n^-(t)\big)^2, \big(c_m^-(t)\big)^2\big\}.
\end{align*}

Comparing now the minimum of $E_\e^+$ and the maximum of $E_\e^-$, we see that under the assumptions of Theorem~\ref{th2.1}, there is a gap $\big(\eta_l(\e), \eta_r(\e)\big)$ in the spectrum of the operator $\Op_\e$ with the properties described in the statement of this theorem.

\subsection*{Acknowledgment}

The reported study by D.B. was funded by RFBR according to the research project 18-01-00046.  The research of P.E. was in part supported by the Czech Science Foundation (GA\v{C}R) within the project 17-01706S and by the European Union within the project CZ.02.1.01/0.0/0.0/16$\underline{\phantom{i}}$019/0000778.



\begin{thebibliography}{999}

\bibitem{Bo15}
D.I.~Borisov: On the band spectrum of a Schr\"odinger operator in a periodic system of domains coupled by small windows, \textit{Russ. J. Math. Phys.} \textbf{22} (2015), 153--160.

\bibitem{Bo16}
D.I.~Borisov: Creation of spectral bands for a periodic domain with small windows, \textit{Russ. J. Math. Phys.} \textbf{23} (2016), 19--34.

\bibitem{Bo18}
D.I. Borisov: On absence of gaps in a lower part of spectrum of Laplacian with frequent alternation of boundary conditions in strip, \textit{Theor. Math. Phys.} \textbf{195} (2018), 690--703.

\bibitem{JPA13}
D.~Borisov, K.~Pankrashkin: Quantum waveguides with small periodic perturbations: gaps and edges of Brillouin zones, \textit{J. Phys. A: Math. Theor.} \textbf{46} (2013), 235203.

\bibitem{BP13a}
D.I.~Borisov, K.V.~Pankrashkin: Gap opening and split band edges in waveguides coupled by a periodic system of small windows, \textit{Math. Notes} \textbf{93} (2013), 660--675.

\bibitem{BP13b}
D.I.~Borisov, K.V.~Pankrashkin: On the extrema of band functions in periodic waveguides, \textit{Funct. Anal. Appl.} \textbf{47} (2013), 238--240.

\bibitem{PRSE} D. Borisov, G. Cardone, T. Durante: Homogenization and uniform resolvent convergence for elliptic operators in a strip perforated along a curve, \textit{Proc. Royal Soc. Edinburgh, Sect A. Math.} \textbf{146} (2016) 1115--1158.

\bibitem{DaTr}
J.~Dahlberg, E.~Trubowitz: A remark on two dimensional periodic potentials, \textit{Comment. Math. Helvetici} \textbf{57} (1982), 130--134.

\bibitem{EK15}
P.~Exner, A.~Khrabustovskyi: On the spectrum of narrow Neumann waveguide with periodically distributed $\delta'$ traps, \textit{J. Phys. A: Math. Theor.} \textbf{48} (2015), 315301.

\bibitem{EK18}
P.~Exner, A.~Khrabustovskyi: Gap control by singular Schr\"odinger operators in a periodically structured metamaterial, \textit{J. Math. Anal. Geom.}, \textbf{14} (2018), 270--285.0


\bibitem{EKW}
P.~Exner, P.~Kuchment, B.~Winn: On the location of spectral edges in $\mathds{Z}$-periodic media, \textit{J. Phys. A: Math. Theor.} \textbf{43} (2010), 474022.

\bibitem{HKSW}
J.~Harrison, P.~Kuchment, A.~Sobolev, B.~Winn: On occurrence of spectral edges for periodic operators inside the Brillouin zone, \textit{J. Phys. A: Math. Theor.} \textbf{40} (2007) 7597--7618.

\bibitem{Il}
A.M. Il'in: \textit{Matching of Asymptotic Expansions of Solutions of Boundary Value Problems}, AMS, Providence, R.I., 1992.

\bibitem{Ka}
Y.E.~Karpeshina: Perturbation theory for the Schr\"odinger operator with a periodic potential, in \textit{Lecture Notes Math}, vol.~1663, Springer 1997

\bibitem{Kh}
A.~Khrabustovskyi: Opening up and control of spectral gaps of the Laplacian in periodic domains, \textit{J. Math. Phys.} \textbf{55} (2014), 121502.

\bibitem{Na}
S.A.~Nazarov: Asymptotic behavior of spectral gaps in a regularly perturbed periodic waveguide, \textit{Vestnik St. Petersburg. Univ. Math.} \textbf{46} (2013), 89--97.

\bibitem{OIS}
O.A. Ole\u{\i}nik,  A.S. Shamaev, G.A. Yosifian: \textit{Mathematical problems in elasticity and homogenization},
Studies in Mathematics and its Applications, vol.~26,  North-Holland, Amsterdam 1992.

\bibitem{Pa}
L.~Parnovski: Bethe-Sommerfeld conjecture, \textit{Ann.H.~Poincar\'{e}} \textbf{9} (2008), 457--508.

\bibitem{PS17}
L.~Parnovski, R.~Shterenberg: Perturbation theory for spectral gap edges of 2D periodic Schr\"odinger operators, \textit{J. Funct. Anal.} \textbf{273} (2017), 444--470.

\bibitem{PaSo}
L.~Parnovski, A.~Sobolev, Bethe-Sommerfeld conjecture for periodic operators with strong perturbations, \textit{Invent. Math.} \textbf{181} (2010), 467--540.

\bibitem{Sk79}
M.M. Skriganov: Proof of the Bethe-Sommerfeld conjecture in dimension two, \textit{Soviet Math. Dokl.} \textbf{20} (1979), 956--959

\bibitem{Sk85}
M.M. Skriganov: The spectrum band structure of the three dimensional Schr\"odinger operator with periodic potential, \textit{Invent. Math.} \textbf{80} (1985), 107--121

\bibitem{SoBe}
A.~Sommerfeld, H.~Bethe: Electronentheorie der Metalle, in \textit{Handbuch der Physik}, 2nd ed., Springer 1933.


\bibitem{V}
O.A. Veliev: Asymptotic formulas for the eigenvalues of a periodic Schr\"odinger operator and the Bethe-Sommerfeld conjecture, \textit{Funct. Anal. Appl.}, \textbf{21} (1987), 87--100.

\bibitem{VL}
M.I. Vishik, L.A. Lyusternik, Regular degeneration and boundary layer for linear differential equations with small parameter, \textit{Uspekhi Mat. Nauk}, {\bf 12}:5(77) (1957), 3--122.

\end{thebibliography}
\end{document}